\documentclass{amsart}
\RequirePackage[utf8]{inputenc}
\usepackage{amsmath,amsthm,amsfonts,amssymb,amscd,amsbsy,dsfont,xfrac,hyperref}
\usepackage[all]{xy}
\usepackage{enumerate}

\newcommand{\Ric}{\operatorname{Ric}}
\newcommand{\scal}{\operatorname{scal}}
\newcommand{\Z}{\mathds Z}

\newcommand{\R}{\mathds R}
\newcommand{\C}{\mathds C}

\newcommand{\Or}{\mathsf O}
\newcommand{\g}{\mathrm g}

\newcommand{\posb}{\sec^\perp>0}
\newcommand{\sign}{\operatorname{sign}}

\allowdisplaybreaks

\hypersetup{
    pdftoolbar=true,
    pdfmenubar=true,
    pdffitwindow=false,
    pdfstartview={FitH},
    pdftitle={Four-dimensional manifolds with positive biorthogonal curvature},
    pdfauthor={Renato G. Bettiol},
    pdfsubject={AMS codes: 53C20, 53C21, 57M50, 57N13, 57R57},
    pdfkeywords={}
    pdfnewwindow=true,
    colorlinks=true, 
    linkcolor=blue,
    citecolor=blue,
    urlcolor=black,
}

\swapnumbers
\newtheorem{theorem}{Theorem}[]

\newtheorem{proposition}[theorem]{Proposition}

\newtheorem{fact}[theorem]{Fact}

\newtheorem*{mainthm}{\sc Theorem}

\theoremstyle{definition}

\theoremstyle{remark}
\newtheorem{remark}[theorem]{Remark}

\title{Four-dimensional manifolds with positive biorthogonal curvature}
\author{Renato G. Bettiol}

\address{\begin{tabular}{l}
University of Pennsylvania\\
Department of Mathematics\\
209 South 33rd St\\
Philadelphia, PA, 19104-6395, USA\\
\emph{E-mail address}: {\tt rbettiol@math.upenn.edu}
\end{tabular}
}

\allowdisplaybreaks
\numberwithin{equation}{section}
\numberwithin{theorem}{section}

\thanks{The author was partially supported by the NSF grant DMS-1209387, USA}

\date{\today}

\begin{document}
\begin{abstract}
We classify, up to homeomorphisms, the closed simply-connected $4$-manifolds that admit a Riemannian metric for which averages of pairs of sectional curvatures of orthogonal planes are positive.
\end{abstract}

\maketitle

\section{Introduction}

Interactions between geometry and topology via curvature restrictions are among the most fascinating connections between fields in Mathematics. Many such interactions flourish in the realm of $4$-manifolds, as can be attested by the vast literature on the subject. However, many longstanding open questions also remain elusive, among which the
classification of closed simply-connected $4$-manifolds that admit a metric with positive sectional curvature ($\sec>0$). Conjecturally, only $S^4$ and $\C P^2$ satisfy this condition, and this is supported by compelling evidence, see~\cite{grove-wilking,hsiang-kleiner}.

As an attempt to better understand $4$-manifolds with $\sec>0$, it is natural to investigate other curvature positivity conditions. Let $(M^4,\g)$ be a Riemannian manifold and, for each plane $\sigma\subset T_pM$, denote by $\sigma^\perp$ its orthogonal plane. The \emph{biorthogonal curvature} of $\sigma$ is defined as the average of sectional curvatures
\begin{equation*}
\sec^\perp(\sigma):=\tfrac12\big(\sec(\sigma)+\sec(\sigma^\perp)\big).
\end{equation*}
The goal of this note is to classify the closed simply-connected $4$-manifolds that admit a metric with positive biorthogonal curvature ($\posb$), more precisely:

\begin{mainthm}
Let $M^4$ be a smoothable closed simply-connected topological $4$-mani\-fold. Up to endowing $M$ with different smooth structures, the following are equivalent:
\begin{enumerate}[\rm (i)]
\item $M^4$ admits a metric with $\sec^\perp>0$;
\item $M^4$ admits a metric with $\Ric>0$;
\item $M^4$ admits a metric with $\scal>0$.
\end{enumerate}
\end{mainthm}

The equivalence between (ii) and (iii) above was established by Sha and Yang~\cite{sha-yang90}. Since scalar curvature is the sum of sectional curvatures of pairs of orthogonal planes, it follows that (i) trivially implies (iii). Thus, in order to prove the above, one must show that (iii) implies (i). The strategy to achieve this consists of constructing metrics with $\posb$ via surgery techniques on closed simply-connected $4$-manifolds homeomorphic to those that satisfy $\scal>0$. Combining the classical work of Donaldson~\cite{donaldson} and Freedman~\cite{freedman} with the well-known $\widehat A$-genus obstruction to $\scal>0$, this amounts to showing that $S^4$, $\#^m\,\C P^2\,\#^n\,\overline{\C P}^2$, and $\#^n (S^2\times S^2)$ admit metrics with $\posb$. In practical terms, this is the list of closed simply-connected $4$-manifolds that satisfy one (and hence all) conditions in the Theorem.

Regarding the building blocks of the above connected sums, the standard metrics on $S^4$ and $\C P^2$ clearly have $\posb$. Furthermore, $S^2\times S^2$ admits metrics with $\posb$ as a consequence of the main result in Bettiol~\cite{biorthogonal}, that actually yields a much more restrictive curvature positivity condition on $S^2\times S^2$. 
The last ingredient to complete the proof of the Theorem is to show that connected sums of manifolds with $\posb$ also admit metrics with $\posb$. This is established using a recent surgery stability criterion of Hoelzel~\cite{hoelzel}, see Proposition~\ref{prop:surgeryposb}.

We remark that simply-connectedness is a necessary hypothesis not only for the equivalence of (ii) and (iii), but also of (i) and (iii). Namely, the standard product metric on $S^3\times S^1$ has $\posb$ (which is essential to prove stability under connected sums), but this manifold has infinite fundamental group and hence does not admit metrics with $\Ric>0$.
We also remark that, since $\#^n\,\C P^2$, as well as $\#^n (S^2\times S^2)$, admit metrics with $\posb$ for all $n$, there is no upper bound on the total Betti number of closed $4$-manifolds with $\posb$. Thus, by the celebrated a priori bounds of Gromov, it follows that there are (many) closed simply-connected $4$-manifolds that admit metrics with $\posb$ but do not admit metrics with $\sec\geq0$. 

Positive biorthogonal curvature was first studied by Seaman~\cite{seaman3,seaman1}, who observed that manifolds with $\sfrac14$-pinched biorthogonal curvature have positive isotropic curvature (and are hence diffeomorphic to a spherical space form). More recently, curvature conditions related to $\posb$ have been studied by Bettiol~\cite{biorthogonal,thesis}, Costa~\cite{ezio}, and Costa and Ribeiro~\cite{ezio2}. Although the latter claims to contain a classification of closed $4$-manifolds with $\sec^\perp\geq0$, a result that would extend the above Theorem, no classification statements are provided. In fact, all results of \cite{ezio2} concern $4$-manifolds satisfying curvature conditions more restrictive than $\posb$.

As our classification of $4$-manifolds with $\posb$ is obtained up to \emph{homeomorphisms}, it is natural to wonder if this can be improved to \emph{diffeomorphisms}. The chief difficulty in carrying this out originates from the (rather serious) difficulty in detecting the diffeomorphism type of $4$-manifolds with $\scal>0$, since the Donaldson and Seiberg-Witten invariants vanish on such $4$-manifolds.

This note is organized as follows. In Section~\ref{sec:psc}, we recall the classification of closed simply-connected $4$-manifolds with $\scal>0$. The key result that $\posb$ is preserved under connected sums is proved in Section~\ref{sec:connected}, completing the proof of the Theorem. The results in this paper comprise part of the author's PhD thesis~\cite{thesis}.

\noindent
\subsection*{Acknowledgement} It is a pleasure to thank my PhD advisor, Karsten Grove, for his thorough and invaluable support along the years, as well as David Wraith for helpful conversations on the surgery stability of curvature positivity conditions.

\section{\texorpdfstring{Four-dimensional manifolds with $\scal>0$}{Four-dimensional manifolds with positive scalar curvature}}\label{sec:psc}

Let $M$ be a closed oriented topological $4$-manifold $M$, and denote by
\begin{equation*}
Q_M\colon H^2(M,\Z)\times H^2(M,\Z)\to\Z, \quad Q_M(\alpha,\beta):=(\alpha\cup\beta)[M],
\end{equation*}
its intersection form, where $[M]\in H_4(M,\Z)\cong\Z$ is the fundamental class of~$M$.
Recall that $Q_M$ is a unimodular symmetric bilinear form, and the intersection form of a connected sum $M_1\# M_2$ is given by $Q_{M_1\# M_2}=Q_{M_1}\oplus Q_{M_2}$.

On the one hand, Freedman~\cite{freedman} proved that every integral symmetric unimodular form $Q$ is realized as the intersection form of a closed simply-connected \emph{topological} $4$-manifold. It also follows from his work that two closed simply-connected smooth $4$-manifolds are homeomorphic if and only if their intersection forms are isomorphic. On the other hand, Donaldson~\cite{donaldson} proved that the only definite forms that can be realized as intersection forms of a \emph{smooth} $4$-manifold are the 
standard diagonal forms $\oplus^m (1)$ and $\oplus^m (-1)$.
Thus, by Serre's classification of indefinite forms, it follows that
a smoothable closed simply-connected topological $4$-manifold must be
homeomorphic to either $S^4$, $\#^m\,\C P^2\,\#^n\,\overline{\C P}^2$, or $\#^{\pm m} M_{E_8}\,\#^n (S^2\times S^2)$, see \cite{donaldson-book}.

Among the above, $S^4$ and $\#^{\pm m} M_{E_8}\,\#^n (S^2\times S^2)$ are spin manifolds. By a classical result of Lichnerowicz, spin manifolds with $\scal>0$ do not admit nontrivial harmonic spinors; and, by the Atiyah-Singer Index Theorem, this implies the vanishing of the $\widehat A$-genus in dimensions multiple of $4$, see \cite{lawson-book}.
Moreover, on a $4$-manifold $M$, the $\widehat A$-genus can be computed in terms of its signature $\sign Q_M=b^+_2(M)-b^-_2(M)$. Namely, the Hirzebruch Signature Theorem yields
\begin{equation*}
\widehat A(M)=-\tfrac18\sign Q_M.
\end{equation*}
Therefore, if $M$ is a closed spin $4$-manifold that admits a metric with $\scal>0$, then $\sign Q_M=0$. Since $\sign Q_{M_1\#M_2}=\sign Q_{M_1}+\sign Q_{M_2}$, $\sign Q_{E_8}=8$ and $\sign Q_{S^2\times S^2}=0$, one concludes the following:

\begin{fact}\label{thm:4manifoldspsc}
A closed simply-connected $4$-manifold that satisfies $\scal>0$ is homeomorphic to either $S^4$, $\#^m\,\C P^2\,\#^n\,\overline{\C P}^2$, or $\#^n (S^2\times S^2)$.
\end{fact}

Conversely, all the above $4$-manifolds are known to admit metrics with $\scal>0$,\newline when endowed with their standard smooth structure. We remark that some of them, such as $\C P^2\#^8 \overline{\C P}^2$, also carry smooth structures without any metrics with $\scal>0$, see \cite{lebrun-catanese}.
Finally, recall that any $M^4=S^4\#^m\,\C P^2\,\#^n\,\overline{\C P}^2\#^p (S^2\times S^2)$ is homeomorphic to one of the manifolds in Fact~\ref{thm:4manifoldspsc}, since $\C P^2\#(S^2\times S^2)$ is diffeomorphic to $\#^2\,\C P^2\,\#\,\overline{\C P}^2$, see \cite{donaldson-book}.

\section{\texorpdfstring{Connected sums of manifolds with $\posb$}{Connected sums of manifolds with positive biorthogonal curvature}}\label{sec:connected}

\begin{proposition}\label{prop:surgeryposb}
If $M_1$ and $M_2$ have metrics with $\posb$, then so does~$M_1\# M_2$.
\end{proposition}

\begin{proof}
The curvature condition $\posb$ corresponds to the cone
\begin{equation*}
C_{\posb}:=\Big\{R\in \mathcal C_B(\R^4): \langle R(\sigma),\sigma\rangle+\langle R(\sigma^\perp),\sigma^\perp\rangle>0 \text{ for all }\sigma\in\mathrm{Gr}_2(\R^4)
\Big\}
\end{equation*}
on the space $\mathcal C_B(\R^4)$ of algebraic curvature operators.\footnote{In other words, $(M^4,\g)$ has $\posb$ if and only if for all $p\in M$ and all linear isometries $\iota\colon\R^4\to T_pM$, the pull-back by $\iota$ of the curvature operator $R_p$ satisfies $\iota^*(R_p)\in C_{\posb}$.}
This is clearly an open $\Or(4)$-invariant convex cone.
Let $R_{S^3\times\R}$ be the curvature operator of the standard product metric on $S^{3}\times\R$. Routine arguments show that $\langle R_{S^3\times\R}(\sigma),\sigma\rangle\geq0$ for all $\sigma\in\mathrm{Gr}_2(\R^4)$ and $\langle R_{S^3\times\R}(\sigma),\sigma\rangle=0$ if and only if $\sigma$ contains the $\R$ direction.
In particular, whenever $\langle R_{S^3\times\R}(\sigma),\sigma\rangle=0$, we have that $\sigma^\perp$ is tangent to $S^3$ and hence $\langle R_{S^3\times\R}(\sigma^\perp),\sigma^\perp\rangle=1$, so $R_{S^3\times\R}\in C_{\posb}$.
The result now follows from the surgery stability criterion of Hoelzel~\cite[Thm.\ B]{hoelzel}.
\end{proof}

As discussed in the Introduction, $S^4$, $\C P^2$ and $S^2\times S^2$ are known to admit metrics with $\posb$. Therefore, by Proposition~\ref{prop:surgeryposb}, all the $4$-manifolds listed in Fact~\ref{thm:4manifoldspsc} admit metrics with $\posb$. This shows that (iii) implies (i), finishing the proof of the Theorem in the Introduction.

Using the equivariant version of the surgery stability criterion of Hoelzel~\cite{hoelzel}, the above metrics with $\posb$ on the manifolds listed in Fact~\ref{thm:4manifoldspsc} can be constructed invariant under an effective circle action, see \cite[Rem.\ 7.18]{thesis} for details. Together with the work of  Baza{\u\i}kin and Matvienko~\cite{baz-mat}, this implies that, in all items in the Theorem in the Introduction, the metric satisfying that curvature condition can be chosen to be invariant under an effective circle action.

\begin{remark}
It follows from the above that $\C P^2\#\overline{\C P}^2$, which is the nontrivial $S^2$-bundle over $S^2$, admits a metric with $\posb$. One such metric is explicitly constructed in \cite{biorthogonal}, using deformation techniques similar to those employed in the construction of metrics with $\posb$ on $S^2\times S^2$.
\end{remark}

\begin{remark}
The condition $\posb$ can be extended to higher dimensions by requiring that averages of pairs of sectional curvatures of any orthogonal planes be positive.
Connected sums of $n$-manifolds with $\posb$ also have $\posb$ as a consequence of \cite[Thm.\ B]{hoelzel}, since $R_{S^{n-1}\times\R}$ belongs to the open $\Or(n)$-invariant convex cone that generalizes the above $C_{\posb}$, see \cite{thesis} for details.
\end{remark}

\begin{remark}
By results of Seaman~\cite{seaman2} and Costa and Ribeiro~\cite{ezio2}, $S^4$ and $\C P^2$ are the only closed simply-connected $4$-manifolds with $\posb$ that can have (weakly) $\sfrac14$-pinched biorthogonal curvature, or nonnegative isotropic curvature,
or satisfy $\sec^\perp\geq\frac{\scal}{24}>0$.
\end{remark}

We conclude with a short discussion of non-simply-connected $4$-manifolds with $\posb$. It is proved in \cite{biorthogonal} that $\R P^2\times\R P^2$ has metrics with $\posb$, and, by the proof of Proposition~\ref{prop:surgeryposb}, so do $S^3\times S^1$ and $(S^3\times\R)/\Gamma$, where $\Gamma$ is a discrete cocompact group. Connected sums of such manifolds also have $\posb$ by Proposition~\ref{prop:surgeryposb}, providing several examples that exhibit a wealth of fundamental~groups.

It is an interesting open problem to determine whether \emph{any} finitely presented group can be realized as the fundamental group of a $4$-manifold with $\posb$, as is the case for $\scal>0$. An affirmative answer would follow from stability of $\posb$ under surgeries of codimension $1$, however the criterion in~\cite{hoelzel} does not apply, as $R_{S^2\times\R^2}\notin C_{\posb}$. 
This suggests that understanding whether $S^2\times T^2$ admits metrics with $\posb$ is crucial in solving the above problem.


\begin{thebibliography}{10}

\bibitem{baz-mat}
{\sc Y.~V. Baza{\u\i}kin and I.~V. Matvienko}, {\em On four-dimensional
  {$T^2$}-manifolds of positive {R}icci curvature}, Sibirsk. Mat. Zh., 48
  (2007), 973--979.

\bibitem{biorthogonal}
{\sc R.~G. Bettiol}, {\em Positive biorthogonal curvature on {$S^2\times
  S^2$}}, Proc. Amer. Math. Soc., 142 (2014), 4341--4353.

\bibitem{thesis}
{\sc R.~G. Bettiol}, {\em On different notions of positivity of curvature}, PhD
  thesis, University of Notre Dame, 2015.

\bibitem{lebrun-catanese}
{\sc F.~Catanese and C.~LeBrun}, {\em On the scalar curvature of {E}instein
  manifolds}, Math. Res. Lett., 4 (1997), 843--854.

\bibitem{ezio}
{\sc E.~Costa}, {\em A modified {Y}amabe invariant and a {H}opf conjecture},
  preprint.
\newblock \htmladdnormallink{arXiv:1207.7107}{http://arxiv.org/abs/1207.7107}.

\bibitem{ezio2}
{\sc E.~Costa and E.~Ribeiro~Jr.}, {\em Four-{D}imensional {C}ompact
  {M}anifolds with {N}onnegative {B}iorthogonal {C}urvature}, Michigan Math.
  J., 63 (2014), 747--761.

\bibitem{donaldson}
{\sc S.~K. Donaldson}, {\em An application of gauge theory to four-dimensional
  topology}, J. Differential Geom., 18 (1983), 279--315.

\bibitem{donaldson-book}
{\sc S.~K. Donaldson and P.~B. Kronheimer}, {\em The geometry of
  four-manifolds}, Oxford Mathematical Monographs, The Clarendon Press, Oxford
  University Press, New York, 1990.
\newblock Oxford Science Publications.

\bibitem{freedman}
{\sc M.~H. Freedman}, {\em The topology of four-dimensional manifolds}, J.
  Differential Geom., 17 (1982), 357--453.

\bibitem{grove-wilking}
{\sc K.~Grove and B.~Wilking}, {\em A knot characterization and 1-connected
  nonnegatively curved 4-manifolds with circle symmetry}, Geom. Topol., 18
  (2014), 3091--3110.

\bibitem{hoelzel}
{\sc S.~Hoelzel}, {\em Surgery stable curvature conditions}, Math.\ Ann., to
  appear.
\newblock \htmladdnormallink{arXiv:1303.6531}{http://arxiv.org/abs/1303.6531}.

\bibitem{hsiang-kleiner}
{\sc W.-Y. Hsiang and B.~Kleiner}, {\em On the topology of positively curved
  {$4$}-manifolds with symmetry}, J. Differential Geom., 29 (1989), 615--621.

\bibitem{lawson-book}
{\sc H.~B. Lawson, Jr. and M.-L. Michelsohn}, {\em Spin geometry}, vol.~38 of
  Princeton Mathematical Series, Princeton University Press, Princeton, NJ,
  1989.

\bibitem{seaman3}
{\sc W.~Seaman}, {\em Existence and uniqueness of algebraic curvature tensors
  with prescribed properties and an application to the sphere theorem}, Trans.
  Amer. Math. Soc., 321 (1990), 811--823.

\bibitem{seaman1}
{\sc W.~Seaman}, {\em Orthogonally pinched curvature tensors and applications},
  Math. Scand., 69 (1991), 5--14.

\bibitem{seaman2}
{\sc W.~Seaman}, {\em On manifolds with nonnegative curvature on totally
  isotropic 2-planes}, Trans. Amer. Math. Soc., 338 (1993), 843--855.

\bibitem{sha-yang90}
{\sc J.-P. Sha and D.~Yang}, {\em Positive {R}icci curvature on compact simply
  connected {$4$}-manifolds}, in Differential geometry: {R}iemannian geometry
  ({L}os {A}ngeles, {CA}, 1990), vol.~54 of Proc. Sympos. Pure Math., Amer.
  Math. Soc., Providence, RI, 1993, 529--538.

\end{thebibliography}
\end{document}